\documentclass[12pt,a4paper]{article}

\usepackage{latexsym}
\usepackage{amsfonts}
\usepackage{amssymb}
\usepackage{amsmath}
\usepackage{amsthm}
\usepackage{mathrsfs}
\usepackage{enumerate}
\usepackage{slashed}

\newcommand{\bwedge}{\raisebox{0.2ex}{${\textstyle \bigwedge}$}}

\numberwithin{equation}{section}

\newtheorem{te}{Theorem}

\newtheorem*{te*}{Theorem}

\begin{document}

\title{The DFR-Algebra for Poisson Vector Bundles}
\author{Michael Forger%
        \thanks{Partly supported by CNPq
                (Conselho Nacional de Desenvolvimento
                 Cient\'{\i}fico e Tecno\-l\'o\-gico), Brazil; \newline
                E-mail: \textsf{forger@ime.usp.br}}
        ~~and~~
        Daniel V.\ Paulino%
        \thanks{Supported by FAPESP
                (Funda\c{c}\~ao de Amparo \`a Pesquisa do
                 Estado de S\~ao Paulo), Brazil; \newline
                 E-mail: \textsf{dberen@ime.usp.br}}
        }
\date{\small{Departamento de Matem\'atica Aplicada \\
      Instituto de Matem\'atica e Estat\'{\i}stica \\
      Universidade de S\~ao Paulo \\
      Caixa Postal 66281 \\
      BR--05314-970~ S\~ao Paulo, S.P., Brazil}}
\maketitle

\thispagestyle{empty}

\begin{abstract}
 \noindent
 The aim of the present paper is to present the construction of a general
 family of $C^*$-algebras that includes, as a special case, the ``quantum
 space-time algebra'' first introduced by Doplicher, Fredenhagen and Roberts.
 To this end, we first review, within the $C^*$-algebra context, the Weyl-Moyal
 quantization procedure on a fixed Poisson vector space (a vector space
 equipped with a given bivector, which may be degenerate).
 We then show how to extend this construction to a Poisson vector bundle
 over a general manifold~$M$, giving rise to a $C^*$-algebra which is also
 a module over $C_0(M)$.
 Apart from including the original DFR-model, this method yields a
``fiberwise quantization'' of general Poisson manifolds.
\end{abstract}

\begin{flushright}
 \parbox{12em}{
  \begin{center}
   Universidade de S\~ao Paulo \\
   RT-MAP-1201 \\
   January 2012
  \end{center}
 }
\end{flushright}

\newpage

\setcounter{page}{1}

\section{Introduction} 
 
In a seminal paper published in 1995~\cite{DFR}, Doplicher, Fredenhagen and
Roberts (DFR) have introduced a special $C^*$-algebra to provide a model for
space-time in which localization of events can no longer be performed with
arbitrary precision: they refer to it as a model of ``quantum space-time''.
Apart from being beautifully motivated, their construction is mathematically
simple: it starts from a symplectic form $\sigma$ on Minkowski space and
considers the corresponding canonical commutation relations (CCRs), which
can be viewed as a representation of a well-known finite-dimensional
nilpotent Lie algebra, the Heisenberg (Lie) algebra.
More precisely, the CCRs appear in the Weyl form, i.e., we are really
dealing with an irreducible unitary representation of the Heisenberg
group.
This in turn generates a $C^*$-algebra that we propose to call
the Heisenberg $C^*$-algebra, related to the original representation
through Weyl quantization, that is, via the Weyl-Moyal star product.
Due to von Neumann's theorem about the uniqueness of the Sch\"{o}dinger
representation, it is unique up to isomorphism.

The main novelty in the DFR construction is that the symplectic form
$\sigma$ defining the Heisenberg algebra is treated as a \emph{variable}.
In this way, one is able to reconcile the construction with the principle
of relativistic invariance: since Minkowski space $\mathbb{R}^4$ has no
distinguished symplectic structure, the only way out is to consider,
simultaneously, \emph{all} possible symplectic structures on Minkowski
space that can be obtained from a fixed one, say $\sigma_0$, by the
action of the Lorentz group.
Now the orbit of $\sigma_0$ under the action of the Lorentz group
is $\, TS^2 \times \mathbb{Z}_2$, thus explaining the origin of
the extra dimensions that appear in this approach.%
\footnote{Note that the factor $\mathbb{Z}_2$ comes from the fact
that we are dealing with the full Lorentz group; it would be absent if
we dropped (separate) invariance under parity $P$ or time reversal $T$.}

Assuming the symplectic form $\sigma$ to vary over the orbit of some
fixed representative $\sigma_0$ produces not just a single Heisenberg
$C^*$-algebra but an entire bundle of $C^*$-algebras over that orbit,
with the Heisenberg $C^*$-algebra for $\sigma_0$ as typical fiber.
The continuous sections of that bundle vanishing at infinity then
define a ``section'' $C^*$-algebra, which carries a natural action
of the Lorentz group induced from its natural action on the underlying
bundle of $C^*$-algebras (which moves base points as well as fibers).
Besides, this ``section'' $C^*$-algebra is also a $C^*$-module over the
``scalar'' $C^*$-algebra $C_0(M)$ of continuous functions on~$M$ vanishing
at infinity.
In the special case considered by DFR, the underlying bundle turns out
to be globally trivial, which in view of von Neumann's theorem implies
a classification result on irreducible as well as on Lorentz covariant
representations of the DFR-algebra.

In retrospect, it is clear that when formulated in this geometrically
inspired language, the results of~\cite{DFR} yearn for generalization~--
even if only for purely mathematical reasons.

From a more physical side, one of the original motivations of the present
work was an idea of J.C.A.~Barata, who proposed to look for a clearer
geo\-metrical interpretation of the classical limit of the DFR-algebra
in terms of coherent states, as developed by K.~Hepp~\cite{HE}.
This led the second author to investigate possible generalizations of
the DFR-construction to other vector spaces than four-dimensional
Minkowski space and other Lie groups than the Lorentz group in four
dimensions.
Isolating the essential hypo\-theses underlying the DFR-construction
from the marginal ones and using the reformulation contained in~%
\cite{GP}, we finally came up with the construction presented in
this paper.%
\footnote{As was kindly pointed out by D.~Bahns, our construction may
 be reduced in some sense to Rieffel's theory of strict deformations,
 as has happened with the original DFR construction~\cite{RFF}.
 Since such a reduction has not yet been worked ou explicitly and
 since the methods used here are quite different, we still believe
 our work may be useful.}
As it turned out, we were even able to eliminate the hypothesis of
nondegeneracy of the underlying symplectic form on each fiber and
thus end up with a version that works even for general Poisson
tensors, which we consider to be a rather dramatic generalization,
taking into account the degree to which Poisson manifolds are more
general than symplectic manifolds.

The original question of how to define the classical limit in the
general context outlined below will be pursued elsewhere, but we
believe that the mathematical construction as presented here is
of independent interest, going beyond the physical motivations
of the original DFR paper.

\section{The Heisenberg $C^*$-Algebra for \\
         Poisson Vector Spaces}

Let $V$ be a Poisson vector space, i.e., a real vector space of
dimension~$n$, say, equipped with a fixed bivector $\sigma$ of
rank $2r$; in other words, the dual $V^*$ of $V$ is a
presymplectic vector space.%
\footnote{Note that we do \emph{not} require $\sigma$ to be
non-degenerate.}
It gives rise to an $(n+1)$-dimensional Lie algebra
$\mathfrak{h}_{\sigma}$ which is a one-dimensional
central extension of the abelian Lie algebra~$V^*$
defined by the cocycle~$\sigma$ and will be called
the \emph{Heisenberg algebra} or, more explicitly,
the \emph{Heisenberg Lie algebra} (associated to $V^*$
and~$\sigma$): as a vector space, $\mathfrak{h}_\sigma
= V^* \oplus \mathbb{R}$, with commutator given by
\begin{equation} \label{eq:HACOMM}
 [(\xi,\lambda),(\eta,\mu)]~=~(0 \,,\, \sigma(\xi,\eta))
 \qquad \mathrm{for}~~\xi,\eta \in V^* \,,\,
 \lambda,\mu \in \mathbb{R}~.
\end{equation}
Associated with this Lie algebra is the \emph{Heisenberg group}
$H_\sigma$: as a manifold, $H_\sigma = V^* \times \mathbb{R}$,
with product (written additively) given by
\begin{equation} \label{eq:HGPROD}
 (\xi,\lambda) \, (\eta,\mu)~=~
 \bigl( \xi+\eta \,,\,
        \lambda + \mu - {\textstyle \frac{1}{2}} \, \sigma(\xi,\eta) \bigr)
 \qquad \mathrm{for}~~\xi,\eta \in V^\ast \,,\,
                      \lambda,\mu \in \mathbb{R}~.
\end{equation}

As is well known, the \emph{canonical commutation relations}
\begin{equation} \label{eq:CCR1}
 \bigl[ \mathbf{P}^j , \mathbf{Q}^k \bigr]~
 = \; - \, i \, \delta^{jk}
 \qquad \mathrm{for}~~1 \leqslant j,k \leqslant r
\end{equation}
can formally be viewed as a representation $\dot{\pi}_\sigma$
of the Heisenberg algebra $\mathfrak{h}_\sigma$ by essentially
skew adjoint operators in a certain Hilbert space, namely
$L^2(\mathbb{R}^{n-r})$, as follows.
First, introduce a \emph{Darboux basis} for~$\sigma$,
that is, a basis $\, \{v_1,\ldots,v_n\} \,$ of~$V$,
with corresponding dual basis $\, \{v^1,\ldots,v^n\} \,$
of~$V^*$, in which $\sigma$ takes the form represented
by the matrix
\begin{equation} \label{eq:STANF1}
 \left( \begin{array}{ccc}
           0  & 1_r & 0 \\
         -1_r &  0  & 0 \\
           0  &  0  & 0
        \end{array} \right) \, ,
\end{equation}
that is, the only non-vanishing matrix elements of~$\sigma$ are
\begin{equation} \label{eq:STANF2}
 \sigma(v^j,v^{r+k})~=~\delta^{jk}~= \; - \, \sigma(v^{r+j},v^k)
 \qquad \mathrm{for}~~1 \leqslant j,k \leqslant r~.
\end{equation}
Such a basis can always be constructed by a symplectic analogue
of Gram-Schmidt ortho\-gonalization; see, for example, \cite%
[Prop.~3.1.2, pp.~162-164]{AM}.
Then, abbreviating $\dot{\pi}_\sigma(\xi,0)$
to $\dot{\pi}_\sigma(\xi)$, set
\begin{equation} \label{eq:REPHA1}
 \begin{array}{c}
  \dot{\pi}_\sigma(v^j)~=~i \, \mathbf{P}^j
  \qquad \mathrm{for}~~1 \leqslant j \leqslant r~, \\[1mm]
  \dot{\pi}_\sigma(v^{r+j})~=~i \, \mathbf{Q}^j
  \qquad \mathrm{for}~~1 \leqslant j \leqslant r~, \\[1mm]
  \dot{\pi}_\sigma(v^{2r+k})~=~i \, \mathbf{Q}^{r+k}
  \qquad \mathrm{for}~~1 \leqslant k \leqslant n-2r~, \\[1mm]
  \dot{\pi}_\sigma(0,1)~= \; i \, 1~,
 \end{array}
\end{equation}
where the $\mathbf{P}$'s and $\mathbf{Q}$'s are the usual
position and momentum operators of quantum mechanics in~%
$L^2(\mathbb{R}^{n-r})$.
Since these operators have a common dense invariant domain of
analytic vectors~\cite{NE} (take, for example, the Schwartz
space~$\mathscr{S}(\mathbb{R}^{n-r})$), this representation can
be exponentiated to a strongly continuous, unitary representation
of the Heisenberg group which will be denoted by $\pi_\sigma$:
these are the canonical commutation relations in \emph{Weyl form}
which, after abbreviating $\pi_\sigma(\xi,0)$ to $\pi_\sigma(\xi)$,
can be written as
\begin{equation} \label{eq:CCR2}
 \pi_\sigma(\xi) \, \pi_\sigma(\eta)~
 =~e^{-\frac{i}{2} \sigma(\xi,\eta)} \, \pi_\sigma(\xi+\eta)~.
\end{equation}
Note that although this explicit definition of~$\dot{\pi}_\sigma$
and (hence) of~$\pi_\sigma$ depends on the choice of Darboux basis,
the resulting representation itself does not, up to unitary
equivalence.
Moreover, both representations are faithful, since they are
restrictions of faithful representations of a larger (non-%
degenerate) Heisenberg algebra/group to a subalgebra/subgroup.

In a second step, we extend this whole construction to the
$C^*$-algebra setting.
To this end, we construct a linear map
\begin{equation} \label{eq:WQUANT1}
 \begin{array}{cccc}
  W_\sigma: & \mathscr{S}(V) & \longrightarrow &
              B(L^2(\mathbb{R}^{n-r})) \\[1mm]
            &       f        &   \longmapsto   &
              W_\sigma f
 \end{array}
\end{equation}
from the space of Schwartz test functions on~$V$ to the
space of bounded linear operators on the Hilbert space~%
$L^2(\mathbb{R}^{n-r})$ by setting
\begin{equation} \label{eq:WQUANT2}
 W_\sigma f~=~\int_{V^*} d\xi~\check{f}(\xi) \; \pi_\sigma(\xi)~,
\end{equation}
which is to be compared with
\begin{equation} \label{eq:FOURT1}
 f(x)~=~\int_{V^*} d\xi~\check{f}(\xi) \; e^{i\langle\xi,x\rangle}~,
\end{equation}
where $\check{f}$ is the inverse Fourier transform of~$f$,
\begin{equation} \label{eq:FOURT2}
 \check{f}(\xi)~=~\frac{1}{(2\pi)^n} \;
                  \int_V dx~f(x) \; e^{-i\langle\xi,x\rangle}~.
\end{equation}
An argument analogous to the one that proves injectivity of the
Fourier transform (on Schwartz space) can be used to show that
$W_\sigma$ is faithful, and an explicit calculation gives
\begin{equation} \label{eq:WMOY1}
 W_\sigma f~W_\sigma g~=~W_\sigma (f \star_\sigma g)
 \qquad \mathrm{for}~~f,g \in \mathscr{S}(V)~,
\end{equation}
where $\star_\sigma$ denotes the \emph{Weyl-Moyal star product},
given by
\begin{equation} \label{eq:WMOY2}
 \begin{array}{c}
 {\displaystyle
  (f \star_\sigma g)(x)~=~\int_{V^*} d\xi~e^{i\langle\xi,x\rangle}
  \int_{V^*} d\eta~\check{f}(\eta) \, \check{g}(\xi-\eta) \;
  e^{-\frac{i}{2} \sigma(\xi,\eta)}} \\[3ex]
  \mathrm{for}~~f,g \in \mathscr{S}(V)~.
 \end{array}
\end{equation}
In other words, using the Weyl-Moyal star product, together with
the standard involution of pointwise complex conjugation, we can
turn $\mathscr{S}(V)$ into a $*$-algebra.

This $*$-algebra admits various norms.
The naive choice would be the sup norm or $L^\infty$ norm,
since obviously, $\mathscr{S}(V) \subset C_0(V)$.
But this is a $C^*$-norm for the usual pointwise product,
not the Weyl-Moyal star product.
However, the above construction shows that there is a natural
$C^*$-norm (which by general nonsense is unique), namely
the operator norm for the above \emph{Weyl quantization map}
$W_\sigma$: the completion of~$\mathscr{S}(V)$ with respect
to this $C^*$-norm will be denoted by $\mathcal{E}_\sigma$.
Observing that this is still a $C^*$-algebra without unit,
the final step consists in passing to its \emph{multiplier
algebra} $\, \mathcal{H}_\sigma = M(\mathcal{E}_\sigma)$:
these are the algebras that we propose to call the \emph%
{Heisenberg $C^*$-algebra}, \emph{without unit} or \emph%
{with unit}, respectively, (associated to $V^*$ and~$\sigma$).
It~is then clear that $\pi_\sigma$ yields an embedding of
the Heisenberg group~$H_\sigma$ into the group of unitaries
of the Heisenberg $C^*$-algebra $\mathcal{H}_\sigma$.

For later use, we note that $\mathcal{E}_\sigma$ contains two
natural dense subspaces which do not depend on~$\sigma$: one
is of course the Schwartz space $\mathscr{S}(V)$ that we started
with, while the other is its completion $\check{L}^1(V)$ with
respect to the $\check{L}^{1}$-norm, that is, the $L^1$-norm
of the inverse Fourier transform: due to the inequalities
$\, \|f\|_\infty \leqslant \|\check{f}\|_1^{} \,$ and
$\, \|W_\sigma f\| \leqslant \|\check{f}\|_1^{}$ \linebreak
which follow immediately from equations~(\ref{eq:FOURT1})
and~(\ref{eq:WQUANT2}), this norm induces a finer topo\-logy
on~$\mathscr{S}(V)$ than that of $C_0(V)$ and also than
that of $\mathcal{E}_\sigma$, so that by completion we get
the inclusions~$\, \check{L}^1(V) \subset C_0(V) \,$ and
$\, \check{L}^1(V) \subset \mathcal{E}_\sigma$.
Moreover, with respect to its own norm (and with respect
to the Weyl-Moyal star product and the standard involution
of pointwise complex conjugation), $\check{L}^1(V)$ is even
a Banach $*$-algebra, and $\mathcal{E}_\sigma$ is its
$C^*$-completion.

Turning to representations of the various mathematical entities
involved, we mention first of all that representations of~%
$\mathcal{H}_\sigma$ correspond uniquely to non-degenerate
representations of~$\mathcal{E}_\sigma$.%
\footnote{A representation of an algebra~$A$ on a vector space~$V$
is called non-degenerate if there is no non-zero vector in~$V$ that
is annihilated by all elements of~$A$. Obviously, if $A$ has a unit,
every representation of~$A$ is non-degenerate. Also, irreducible
representations are always non-degenerate. Finally, it can be
proved that any non-degenerate representation of an algebra~$A$
extends uniquely to a representation of its multiplier algebra
$M(A)$.}
Similarly, strongly continuous unitary representations
of~$H_\sigma$ correspond uniquely to representations of~%
$\mathfrak{h}_\sigma$ which in~\cite{DFR} are called regular:
according to Nelson's theo\-rem~\cite{NE}, these are
representations by essentially skew adjoint operators
with a common dense invariant domain of analytic vectors.
Finally, it is obvious that the former induce the latter
(just by restriction from~$\mathcal{H}_\sigma$ to~$H_\sigma$),
while the converse follows from an argument similar to one
already used above: given a strongly continuous unitary
representation~$\pi$  of~$H_\sigma$ on some Hilbert space
$\mathfrak{H}$, the linear map
\begin{equation} \label{eq:WQUANT3}
 \begin{array}{cccc}
  W_\pi: & \mathscr{S}(V) & \longrightarrow &
           B(\mathfrak{H}) \\[1mm]
         &       f        &   \longmapsto   &
           W_\pi f
 \end{array}
\end{equation}
defined by
\begin{equation} \label{eq:WQUANT4}
 W_\pi f~=~\int_{V^*} d\xi~\check{f}(\xi) \; \pi(\xi)~,
\end{equation}
extends to a representation of~$\mathcal{E}_\sigma$ which,
using an argument similar to the one that guarantees
injectivity of the map~(\ref{eq:WQUANT1}), can be shown to be
non-degenerate.

All these constructions seem to be well known when $\sigma$
is \emph{non-degenerate}: in that case, the representation
$\pi_\sigma$ is \emph{irreducible} and, according to one of
von Neumann's famous theorems, is the \emph{unique} such
representation, generally known as the \emph{Schr\"odinger
representation of the canonical commutation relations}.

In the degenerate case, i.e., when $\sigma$ has a non-trivial null
space, denoted by $\, \ker \sigma$, we can use the following trick:
choose a subspace $V'$ of~$V^*$ complementary to $\, \ker \sigma$,
so that the restriction $\sigma'$ of~$\sigma$ to~$\, V' \times V' \,$
is non-degenerate, and introduce the corresponding Heisenberg algebra
$\, \mathfrak{h}_{\sigma'} = V' \oplus \mathbb{R} \,$ and Heisenberg
group $\, H_{\sigma'} = V' \times \mathbb{R} \,$ to decompose the
original ones into the direct sum $\, \mathfrak{h}_\sigma = \ker
\sigma \oplus \mathfrak{h}_{\sigma'} \,$ of two commuting ideals
and $\, H_\sigma = \ker \sigma \times H_{\sigma'} \,$ of two
commuting normal subgroups.%
\footnote{As is common practice in the abelian case, we consider
the same vector space $\, \ker \sigma$ as an abelian Lie algebra
in the first case and as an (additively written) abelian Lie group
in the second case, so that the exponential map becomes the identity.}
It follows that every (strongly continuous unitary) representation
of~$H_\sigma$ is the tensor product of a (strongly continuous unitary)
representation of~$\, \ker \sigma$ and a (strongly continuous unitary)
representation of~$H_{\sigma'}$, where the first is irreducible if and
only if each of the last two is irreducible.
Now since $\, \ker \sigma$ is abelian, its irreducible representations
are one-dimensional and given by their character, which proves the
following
\begin{te}[Classification of irreducible representations] \label{vr}
 With the notation above, the strongly continuous, unitary, irreducible
 representations of the Heisenberg group~$H_\sigma$, or equivalently, the
 irreducible representations of the Heisenberg $C^*$-algebras without
 unit\/~$\mathcal{E}_\sigma$ or with unit\/~$\mathcal{H}_\sigma$, are classified
 by their \textbf{highest weight}, which is a vector $v$ in~$V$, or more
 precisely, its class~$[v]$ in the quotient space $V/(\ker \sigma)^\perp$,
 such that
 \[
  \pi_{[v]}(\xi,\eta)~=~e^{i\langle\xi,v\rangle} \, \pi_{\sigma'}(\eta)
  \qquad \mathrm{for}~~\xi \in \ker \sigma \,,\,
         \eta \in H_{\sigma'}~,
 \]
 where $\pi_{\sigma'}$ is of course the Schr\"odinger representation of~%
 $H_{\sigma'}$.
\end{te}
\noindent
It should be noted, however, that the representation $\pi_\sigma$ used to
construct the Heisenberg $C^*$-algebras is of course very far from being
irreducible.

\section{The DFR-Algebra for Poisson Vector Bundles}

Let $E$ be a Poisson vector bundle, i.e., a (smooth) real vector bundle of
fiber dimension~$n$, say, over a (smooth) manifold~$M$, with typical fiber~%
$\mathbb{E}$, equipped with a fixed (smooth) bivector field $\sigma$; in other
words, the dual $E^*$ of $E$ is a (smooth) presymplectic vector bundle.%
\footnote{Note that we do \emph{not} require $\sigma$ to be non-degenerate
or even to have constant rank.}
Then it is clear that we can apply all constructions of the previous section
to each fiber.
The question to be addressed in this section is how the results can be glued
together along the base manifold~$M$ and to describe the resulting global
objects.

Starting with the collection of Heisenberg algebras $\mathfrak{h}_{\sigma(m)}$
($m \in M$), we note first of all that these fit together into a (smooth) real
vector bundle over~$M$, which is just the direct sum of~$E^*$ and the trivial
line bundle $\, M \times \mathbb{R} \,$ over~$M$.
The non-trivial part is the commutator, which is defined by equation~%
(\ref{eq:HACOMM}) applied to each fiber, turning this vector bundle into
a \emph{totally intransitive Lie algebroid} \cite[Def.~3.3.1, p.~100]{MK}
which we shall call the \emph{Heisenberg algebroid} associated to~%
$(E,\sigma)$ and denote by $\mathfrak{h}(E,\sigma)$: it will even be
a \emph{Lie algebra bundle} \cite[Def.~3.3.8, p.~104]{MK} if and only
if $\sigma$ has constant rank.
Of course, spaces of sections (with certain regularity properties)
of~$\mathfrak{h}(E,\sigma)$ will then form (infinite-dimensional)
Lie algebras with respect to the (pointwise defined) commutator,
but the correct choice of regularity conditions is a question of
functional analytic nature to be dictated by the problem at hand.

Similarly, considering the collection of Heisenberg groups~$H_{\sigma(m)}$
($m \in M$), we note that these fit together into a (smooth) real fiber
bundle over~$M$, which is just the fiber product of~$E^*$ and the trivial
line bundle $M \times \mathbb{R}$.
Again, the non-trivial part is the product, which is defined by equation~%
(\ref{eq:HGPROD}) applied to each fiber, turning this fiber bundle into
a \emph{totally intransitive Lie groupoid} \cite[Def.~1.1.3,
p.~5 \& Def.~1.5.9, p.~32]{MK} which we shall call the \emph{Heisenberg
groupoid} associated to~$(E,\sigma)$ and denote by $H(E,\sigma)$: it will
even be a \emph{Lie group bundle} \cite[Def.~1.1.19, p.~11]{MK} if and
only if $\sigma$ has constant rank.
And again, spaces of sections (with certain regularity properties)
of~$H(E,\sigma)$ will form (infinite-dimensional) Lie groups with
respect to the (pointwise defined) product, but the correct choice
of regularity conditions is a question of functional analytic nature
to be dictated by the problem at hand.

An analogous strategy can be applied to the collection of Heisenberg
$C^*$-algebras $\mathcal{E}_{\sigma(m)}$ and $\mathcal{H}_{\sigma(m)}$
($m \in M$), but the details are somewhat intricate since the fibers
are now (infinite-dimensional) $C^*$-algebras which may depend on the
base point in a discontinuous way, since the rank of~$\sigma$ is allowed
to jump.
Concretely, our goal is to fit the collections of Heisenberg $C^*$-algebras
$\mathcal{E}_{\sigma(m)}$ and $\mathcal{H}_{\sigma(m)}$ into continuous fields
of $C^*$-algebras over~$M$ \cite[Def.~10.3.1, p.~218]{DI},
\cite[Def.~2.10, pp.~68-69]{GBVF} whose continuous sections (subject to
appropriate decay or boundedness conditions at infinity) yield spaces
$\mathscr{E}(E,\sigma)$ and $\mathscr{H}(E,\sigma)$ \linebreak

\pagebreak

\noindent
which are not only again $C^*$-algebras but also modules over the function
algebras $C_0(M)$ and/or~$C_b(M)$.%
\footnote{As usual, we denote by $C_0(M)$ the commutative $C^*$-algebra
of continuous functions on~$M$ that vanish at infinity and by $C_b(M)$
the commutative $C^*$-algebra of bounded continuous functions on~$M$;
as is well known, the latter is the multiplier algebra of the former.}

To do so, we start by noting that there is a naturally defined smooth
vector bundle over~$M$ which we shall denote by $\check{\mathcal{L}}^1$:
its fibers are just the Banach spaces of $\check{L}^1$-functions on the
fibers of~$E$, i.e.
\begin{equation} \label{eq:BANBUN}
 \check{\mathcal{L}}^1~= \bigcup_{m \in M}^{\scriptscriptstyle{\bullet}}
                         \check{\mathcal{L}}_m^1
 \qquad \mathrm{where} \qquad
 \check{\mathcal{L}}_m^1~=~\check{L}^1(E_m)~.
\end{equation}
(See~\cite{LA} for the theory of vector bundles whose typical fiber
is a fixed Banach space.)
Note that as a vector bundle, $\check{\mathcal{L}}^1$ does \emph{not}
depend on $\sigma$ and is locally trivial, but of course, the fiberwise
Weyl-Moyal star product on this vector bundle does depend on~$\sigma$
and thus will be locally trivial, turning $\check{\mathcal{L}}^1$ into
a smooth bundle of Banach $*$-algebras, if and only if $\sigma$ has
constant rank.%
\footnote{This is exactly the same situation as for the fiberwise
commutator in the Heisenberg algebroid or the fiberwise product in
the Heisenberg grupoid.}

But at any rate, it is clear that $\check{\mathcal{L}}^1$ is a
continuous field of Banach $*$-algebras over~$M$ and that the
space $C_0(\check{\mathcal{L}}^1)$ of continuous sections of~%
$\check{\mathcal{L}}^1$ vanishing at infinity is a Banach
$*$-algebra which we shall denote by $\mathscr{E}^0(E,\sigma)$,
with norm given by
\begin{equation} \label{eq:DFRN1}
 \| \varphi \|_{\mathscr{E}^0}~
 = \; \sup_{m \in M} \| \varphi(m) \|_{\check{L}^1(E_m)}
 \qquad \mathrm{for}~~\varphi \in \mathscr{E}^0(E,\sigma)~.
\end{equation}
Equally clear is that $\mathscr{E}^0(E,\sigma)$ is not only a
Banach $*$-algebra but also a module over the function algebra
$C_0(M)$ (and even over the function algebra~$C_b(M)$), with
\begin{equation} \label{eq:DFRMOD1}
 \| f \varphi \|_{\mathscr{E}^0}~
 \leqslant \; \|f\|_\infty \, \| \varphi \|_{\mathscr{E}^0}
 \qquad \mathrm{for}~~f \in C_b(M) \,,\, \varphi \in \mathscr{E}^0(E,\sigma)~,
\end{equation}
and for every point $m$ in~$M$, the evaluation map (Dirac delta function)
at~$m$,
\begin{equation} \label{eq:DELTA1}
 \begin{array}{cccc}
  \delta_m : & \mathscr{E}^0(E,\sigma) = C_0(\check{\mathcal{L}}^1)
             & \longrightarrow & \check{L}^1(E_m) \\[1mm]
             & \varphi
             &   \longmapsto   &    \varphi(m)
 \end{array}
\end{equation}
is a Banach $*$-algebra homomorphism which, due to local triviality, is onto.
Composing it with the Weyl quantization map $W_{\sigma(m)}$ (naturally extended
from $\mathscr{S}(E_m)$ to $\check{L}^1(E_m)$, by continuity) gives
a $*$-representation $\, W_m = W_{\sigma(m)} \circ \delta_m \,$ of~%
$\mathscr{E}^0(E,\sigma)$ by bounded operators on some Hilbert space,
and since the family of all these $*$-representations is separating
(i.e., $\bigcap_{m \in M} \ker W_m = \{0\}$), it provides a $C^*$-norm
on $\mathscr{E}^0(E,\sigma)$, explicitly given by
\begin{equation} \label{eq:DFRN2}
 \begin{array}{c}
  {\displaystyle
   \|\varphi\|_{\mathscr{E}}^{}~
   = \; \sup_{m \in M} \| \varphi(m) \|_{\mathcal{E}_{\sigma(m)}}^{}~
   = \; \sup_{m \in M} \| W_{\sigma(m)}(\varphi(m)) \|} \\[2ex]
  \mathrm{for}~~\varphi \in \mathscr{E}^0(E,\sigma)~.
 \end{array}
\end{equation}
The completion of~$\mathscr{E}^0(E,\sigma)$ with respect to
this $C^*$-norm will be denoted by $\mathscr{E}(E,\sigma)$,
and we have
\begin{equation} \label{eq:DFRN3}
 \|\varphi\|_{\mathscr{E}}^{}~\leqslant~\|\varphi\|_{\mathscr{E}^0}
 \qquad \mathrm{for}~~\varphi \in \mathscr{E}^0(E,\sigma)~,
\end{equation}
as well as
\begin{equation} \label{eq:DFRN4}
 \begin{array}{c}
  {\displaystyle
   \|\varphi\|_{\mathscr{E}}^{}~
   = \; \sup_{m \in M} \| \varphi(m) \|_{\mathcal{E}_{\sigma(m)}}^{}~
   = \; \sup_{m \in M} \| W_{\sigma(m)}(\varphi(m)) \|} \\[2ex]
  \mathrm{for}~~\varphi \in \mathscr{E}(E,\sigma)~.
 \end{array}
\end{equation}
Again, $\mathscr{E}(E,\sigma)$ is not only a $C^*$-algebra but also a
module over the function algebra $C_0(M)$ (and even over the function
algebra~$C_b(M)$), with
\begin{equation} \label{eq:DFRMOD2}
 \| f \varphi \|_{\mathscr{E}}~
 \leqslant \; \|f\|_\infty \, \| \varphi \|_{\mathscr{E}}
 \qquad \mathrm{for}~~f \in C_b(M) \,,\, \varphi \in \mathscr{E}(E,\sigma)~,
\end{equation}
and for every point $m$ in~$M$, the evaluation map (Dirac delta function)
at~$m$,
\begin{equation} \label{eq:DELTA2}
 \begin{array}{cccc}
  \delta_m : & \mathscr{E}(E,\sigma)
             & \longrightarrow & \mathcal{E}_{\sigma(m)} \\[1mm]
             & \varphi
             &   \longmapsto   &      \varphi(m)
 \end{array}
\end{equation}
is a $C^*$-algebra homomorphism which, having a dense image, is onto.%
\footnote{These statements are easily derived by noting that, for any
$\, f \in C_b(M) \,$ and any point $m$ in $M$, multiplication by~$f$ as
a linear map from~$\mathscr{E}^0(E,\sigma)$ to~$\mathscr{E}^0(E,\sigma)$
and evaluation at~$m$ as a linear map from~$\mathscr{E}^0(E,\sigma)$ to~%
$\check{L}^1(E_m)$ are continuous not only in their own topologies,
but, according to equation~(\ref{eq:DFRN2}), also with respect to
their $C^*$-norms: this implies that they admit the required unique
continuous linear extensions.}
Finally, observing that $\mathscr{E}(E,\sigma)$ is still a $C^*$-algebra
without unit, the final step consists in passing to its multiplier algebra
$\, \mathscr{H}(E,\sigma) = M(\mathscr{E}(E,\sigma))$: then, once more,
$\mathscr{H}(E,\sigma)$ is not only a $C^*$-algebra but also a module
over the function algebra $C_0(M)$ (and even over the function
algebra~$C_b(M)$), with
\begin{equation} \label{eq:DFRMOD3}
 \| f \varphi \|_{\mathscr{H}}~
 \leqslant \; \|f\|_\infty \, \| \varphi \|_{\mathscr{H}}
 \qquad \mathrm{for}~~f \in C_b(M) \,,\, \varphi \in \mathscr{H}(E,\sigma)~.
\end{equation}
and for every point $m$ in~$M$, the evaluation map (Dirac delta function)
at~$m$,
\begin{equation} \label{eq:DELTA3}
 \begin{array}{cccc}
  \delta_m : & \mathscr{H}(E,\sigma)
             & \longrightarrow & \mathcal{H}_{\sigma(m)} \\[1mm]
             & \varphi
             &   \longmapsto   &      \varphi(m)
 \end{array}
\end{equation}
is a $C^*$-algebra homomorphism which, having a dense image, is onto.%
\footnote{The argumentation that permits these extensions remains to
be completed.}

Noting that the module structure and the evaluation maps are related by
the obvious formula
\begin{equation} \label{eq:DFRH1}
 \begin{array}{c}
  \delta_m(f \varphi)~=~f(m) \, \delta_m(\varphi) \\[1ex]
  \mathrm{for}~~f \in C_0(M)~~\mathrm{and}~~\varphi \in
  \mathscr{E}^0(E,\sigma)~\mathrm{or}~\mathscr{E}(E,\sigma)~
                          \mathrm{or}~\mathscr{H}(E,\sigma)~,
 \end{array}
\end{equation}
we see that $(\mathscr{E}^0(E,\sigma),(\delta_m)_{m \in M},M)$ is a
field of Banach $*$-algebras over~$M$ whereas $(\mathscr{E}(E,\sigma),
(\delta_m)_{m \in M},M)$ and $(\mathscr{H}(E,\sigma),(\delta_m)_{m \in M},M)$
are fields of $C^*$-algebras over~$M$.
As we have seen before, the first is certainly a \emph{continuous} field
of Banach $*$-algebras, while the question whether the last two are also
\emph{continuous} fields of $C^*$-algebras, which would provide the disjoint
unions
\begin{equation} \label{eq:CSTBUN}
 \mathcal{E}~= \bigcup_{m \in M}^{\scriptscriptstyle{\bullet}}
               \mathcal{E}_{\sigma(m)}
 \qquad \mathrm{and} \qquad
 \mathcal{H}~= \bigcup_{m \in M}^{\scriptscriptstyle{\bullet}}
               \mathcal{H}_{\sigma(m)}
\end{equation}
with natural topologies so as to turn them into total spaces of Fell
bundles such that $\mathscr{E}(E,\sigma)$ becomes the space of continuous
sections of $\mathcal{E}$ that vanish at infinity and $\mathscr{H}(E,\sigma)$
becomes the space of bounded continuous sections of~$\mathcal{H}$~%
\cite[Sect.~II.13.18, pp.~132-134]{FD}, is still under investigation.

Inspired by the construction in \cite{GP}, we propose to call the
$C^*$-algebra modules $\mathscr{E}(E,\sigma)$ and $\mathscr{H}(E,\sigma)$
the \emph{DFR-algebra} (associated to $(E,\sigma)$), \emph{without unit}
or \emph{with unit}, respectively.

Concerning irreducible representations, we have
\begin{te}
 Every irreducible representation of~$\mathscr{E}(E,\sigma)$, or equivalently,
 of~$\mathscr{H}(E,\sigma)$, is of the form $\, \pi_{[v_m]} \circ \delta_m^{}$,
 where $m$ is some point in~$M$ and $[v_m]$ is the class of a vector $\, v_m
 \in E_m \,$ in the quotient space $\, E_m/\ker \sigma(m)^\perp$, as in
 Theorem~\ref{vr}.
\end{te}
\begin{proof}
 Let $\pi$ be a irreducible representation of $\mathscr{E}(E,\sigma)$, which
 extends uniquely to an irreducible representation of its multiplier algebra
 $\mathscr{H}(E,\sigma)$.
 Then since $\mathscr{E}(E,\sigma)$ is a module over $C_b(M)$, we have a
 canonical $C^*$-algebra homomorphism
 \[
  C_b(M)~\longrightarrow~Z \bigl( \mathscr{H}(E,\sigma) \bigr)
 \]
 of~$C_b(M)$ into the center $Z \bigl( \mathscr{H}(E,\sigma) \bigr)$ of~%
 $\mathscr{H}(E,\sigma)$.
 Thus $\pi$ restricts to an irreducible representation of $C_b(M)$, which
 must be one-dimensional and given by a character.
 But the characters of~$C_b(M)$ are the points of~$M$, so there is a point
 $\, m \in M \,$ such that
 \[
  \pi(f)~=~f(m) \, 1 \qquad \mathrm{for}~~f \in C_b(M)~,
 \]
 and hence
 \[
  \pi(f \varphi)~=~f(m) \, \pi(\varphi)
  \qquad \mathrm{for}~~f \in C_b(M) \,,\, \varphi \in \mathscr{E}(E,\sigma)~,
 \]
 which implies that $\pi$ must vanish on the kernel of $\delta_m$.
 Therefore, $\pi$ factors over $\ker \delta_m$, that is, $\pi =
 \pi_m \circ \delta_m \,$ where $\pi_m$ is an irreducible
 representation of $\mathcal{E}_{\sigma(m)}$.
 Now apply Theorem~\ref{vr}.
\end{proof}

\section{Examples}

\subsection{Homogeneous Vector Bundles}

A first important special case of the construction outlined in the previous
section occurs when the underlying manifold~$M$ and Poisson vector bundle~%
$(E,\sigma)$ are homogeneous.
More specifically, assume that $G$ is a Lie group which acts transitively
(and properly) on~$M$ as well as on~$E$ and such that $\sigma$ is $G$-%
invariant: this means that writing
\begin{equation}
 \begin{array}{ccc}
  G \times M & \longrightarrow &     M     \\[1mm]
     (g,m)   &   \longmapsto   & g \cdot m
 \end{array}
 \qquad \mathrm{and} \qquad
 \begin{array}{ccc}
  G \times E & \longrightarrow &     E     \\[1mm]
     (g,u)   &   \longmapsto   & g \cdot u
 \end{array}
\end{equation}
for the respective actions, where the latter is linear along the fibers
and hence induces an action
\begin{equation}
 \begin{array}{ccc}
  G \times \bwedge^2 E & \longrightarrow & \bwedge^2 E \\[1mm]
          (g,u)        &   \longmapsto   &  g \cdot u
 \end{array}
\end{equation}
so that
\begin{equation}
 \sigma(g \cdot m)~=~g \cdot \sigma(m)
 \qquad \mathrm{for}~~g \in G \,,\, m \in M~.
\end{equation}
Choosing a reference point $m_0$ in~$M$ and denoting by $H$ the stability
group of~$m_0$ in~$G$, by $\mathbb{E}$ the fiber of~$E$ over $m_0$ and by
$\sigma_0$ the value of the bivector field $\sigma$ at~$m_0$, we can
identify $M$ with the homogeneous space $G/H$, $E$ with the vector
bundle $\, G \times_H \mathbb{E} \,$ associated to~$G$, viewed as a
principal $H$-bundle over~$G/H$, and the representation of~$H$ on~%
$\mathbb{E}$ obtained from the action of~$G$ on~$E$ by appropriate
restriction, and $\sigma$ with the bivector field obtained from
$\sigma_0$ by the association process.
Explicitly, for example, we identify the left coset $\, gH \in G/H \,$
with the point $\, g \cdot m_0 \in M \,$ and the equivalence class~%
$\, [g,u_0] = [gh,h^{-1} \cdot u_0] \in G \times_H \mathbb{E}$ \linebreak
with the vector $\, g \cdot u_0 \in E$.
As a result, we see that if the representation of~$H$ on~$\mathbb{E}$
extends to a representation of~$G$, then the associated bundle 
$\, G \times_H \mathbb{E} \,$ is globally trivial: an explicit
trivialization is given by
\begin{equation}
 \begin{array}{ccc}
       G \times_H \mathbb{E}      & \longrightarrow & M \times \mathbb{E}
  \\[1mm]
  [g,u_0] = [gh,h^{-1} \cdot u_0] &   \longmapsto   & (gH,g^{-1} \cdot u_0)
 \end{array}
\end{equation}
Of course, in this case, $G$-invariance implies that $\sigma$ has
constant rank and hence the Heisenberg algebroid becomes a Lie algebra
bundle, the Heisenberg groupoid becomes a Lie group bundle and the
fields of $C^*$-algebras $(\mathscr{E}(E,\sigma),(\delta_m)_{m \in M},M)$
and $(\mathscr{H}(E,\sigma),(\delta_m)_{m \in M},M)$ over~$M$ are
associated to locally trivial bundles of $C^*$-algebras $\mathcal{E}$
and $\mathcal{H}$ over~$M$ (see equation~(\ref{eq:CSTBUN})).
And if the representation of~$H$ on~$\mathbb{E}$ extends to a
representation of~$G$, all these bundles will even be globally
trivial.

This special situation prevails in the case of the original DFR-algebra,
where $\mathbb{E}$ is four-dimensional Minkowski space $\mathbb{R}^{1,3}$,
$G$ is (possibly up to a two-fold covering) the Lorentz group $\mathrm{O}%
(1,3)$, $H$ is its intersection with the symplectic group $\mathrm{Sp}%
(4,\mathbb{R})$, so~$H$ is an abelian Lie group of type $\, \mathbb{R}
\times \mathrm{U}(1)$, and $\sigma_0$ is the standard symplectic form,
defined by the matrix $\bigl( \begin{smallmatrix} 0 & 1_2 \\ -1_2 & 0
\end{smallmatrix} \bigr)$.

\subsection{Poisson Manifolds} 

Let $M$ be a Poisson manifold with Poisson tensor~$\sigma$.
In this case, it is natural to assume $E$ to be the tangent
bundle $TM$ of~$M$.
Here, we have one piece of additional information: the requirement
that the Poisson tensor should be integrable (i.e., that the Schouten
bracket of~$\sigma$ with itself should vanish) implies that $M$ admits
a \emph{foliation} into symplectic submanifolds, called its \emph%
{symplectic leaves}~\cite{SF}.
Then since the Poisson tensor has constant (and maximal) rank along
each leaf, it is clear from the construction of the previous section
that, upon restriction to each leaf $S$, the Heisenberg algebroid
becomes a Lie algebra bundle, the Heisenberg groupoid becomes a Lie
group bundle and the fields of $C^*$-algebras $(\mathscr{E}(TM|_S,
\sigma_S),(\delta_s)_{s \in S},S)$ and $(\mathscr{H}(TM|_S,\sigma_S),
(\delta_s)_{s \in S},S)$ over~$S$ are associated to locally trivial
bundles of $C^*$-algebras $\mathcal{E}_S$ and $\mathcal{H}_S$
over~$S$.

\section{Outlook}

Our first goal when starting this investigation was to find an appropriate
mathematical setting for geometrical generalizations of the DFR model.
Here, we report on first results in this direction, but there are of
course various steps that still have to be taken, for example:
\begin{itemize}
 \item Establish closer contact between the Heisenberg Lie algebroid and
       Heisenberg groupoid on the one hand and the fields of $C^*$-algebras
       that give rise to the DFR-algebra on the other hand.
       This will most likely require extending the concepts of unitaries
       and of affiliated unbounded elements from $C^*$-algebras to fields
       of $C^*$-algebras and/or $C^*$-modules over $C_0(M)$ and/or $C_b(M)$.
 \item The same goes for the concept of states.
 \item Another important question is how this construction relates to
       \emph{deformation quantization}, both with regard to its formal
       version, such as Fedosov's construction for symplectic manifolds
       or Kontsevich's theorem for Poisson maifolds, and with regard to
       Rieffel's strict deformation quantization.
\end{itemize}
We are fully aware of the fact that these questions are predominantly of
mathematical nature: the physical interpretation is quite another matter.
But to a certain extent this applies even to the original DFR-model, since
it is not clear how to extend the interpretation of the commutation
relations postulated in~\cite{DFR}, in terms of uncertainty relations,
to other space-time manifolds, or even to Minkowski space in dimensions
$\neq 4$.
In addition, it should not be forgotten that, even classically, space-time
coordinates are \emph{not} observables: this means that the basic axiom of
algebraic quantum field theory according to which observables should be
described by (local) algebras of a certain kind (such as $C^*$-algebras
or von Neumann algebras) does \emph{not at all} imply that in quantum
gravity one should replace classical space-time coordinate functions
by noncommuting operators.
To us, the basic question seems to be: \emph{How can we formulate space-time
uncertainty relations, in the sense of obstructions to the possibility of
localizing events with arbitrary precision, in terms of \textbf{observables}?}
That of course stirs up the question:
\emph{How do we actually \textbf{measure} the geometry of space-time when
quantum effects become strong?}

\section*{Acknowledgments}

We would like to thank J.C.A.~Barata for suggesting the original problem
of extending the DFR model to higher dimensions and analyzing its classical
limit using coherent states, which provided the motivation for all the
developments reported here.
We are also grateful to P.L.~Ribeiro for his many valuable suggestions
and remarks.









\begin{thebibliography}{99}

\bibitem{AM} Abraham, R., Marsden, J.E.,
 \emph{Foundations of Mechanics}, $2^{\mathrm{nd}}$ edition,
 Benjamin-Cummings, Reading 1978.

\bibitem{DI} Dixmier, J.,
 \emph{C*-Algebras},
 North-Holland, Amsterdam 1977.

\bibitem{DFR} Doplicher S., Fredenhagen K., Roberts J.E.,
 \emph{The Quantum Structure of Spacetime at the Planck Scale
       and Quantum Fields},
 Commun.\ Math.\ Phys.\ \textbf{172} (1995) 187-220, hep-th/0303037.

\bibitem{FD} Fell, J.M.G., Doran, R.S.,
 \emph{Representations of *-Algebras, Locally Compact Groups
       and Banach *-Algebraic Bundles},
 Vol.~1: \emph{Basic Representation Theory of Groups and Algebras},
 Vol.~2: \emph{Banach *-Algebraic Bundles, Induced Representations,
               and Generalized Mackey Analysis},
 Academic Press, San Diego 1988.

\bibitem{GBVF} Gracia-Bondía, J.M., Várilly, J.C., Figueroa, H.,
 \emph{Elements of Noncommutative Geometry},
 Birkh\"auser, Basel 2001.

\bibitem{HE} Hepp, K.,
 \emph{The Classical Limit for Quantum Mechanical Correlation Functions},
 Commun.\ Math.\ Phys.\ \textbf{35} (1974) 265-277.


\bibitem{LA} Lang, S.,
 \emph{Differential and Riemannian Manifolds},
 Springer, Berlin 1995.

\bibitem{MK} Mackenzie, K.C.H.,
 \emph{General Theory of Lie Groupoids and Lie Algebroids},
 Cambridge University Press, Cambridge 2005.

\bibitem{MU} Murphy, G.J.,
 \emph{$C^*$-Algebras and Operator Theory},
 Academic Press, New York 1990.

\bibitem{NE} Nelson, E.,
 \emph{Analytic Vectors},
 Ann.\ Math.\ \textbf{70} (1959) 572-615.

\bibitem{GP} Piacitelli G.,
 \emph{Quantum Spacetime: A Disambiguation},
 SIGMA \textbf{6} (2010) 073, 43 pp.,
 arXiv:math-ph/1004.5261v3.

\bibitem{RFF} Rieffel, M.A.,
 \emph{On the Operator Algebra for the Space-Time Uncertainty Relations},
 in: Proceedings of the Rome Conference on Operator Algebras and Quantum
 Field Theory, 1996, pp.~374-382,
 Int. Press, Cambridge, MA 1997,
 arXiv:funct-an/9701011.

\bibitem{SF} Weinstein, A.,
 \emph{The Local Structure of Poisson Manifolds},
 J.\ Diff.\ Geom.\ \textbf{18} (1983) 523-557.

\end{thebibliography}
\end{document}